\documentclass{amsart}
\usepackage{latexsym, bbm, html}
\usepackage[full, dcucite]{harvard}
\usepackage[dvips]{graphicx}
\usepackage{setspace}

\newtheorem{Corollary}{Corollary}
\newtheorem{Proposition}{Proposition}
\newtheorem{Lemma}{Lemma}

\newtheorem{Theorem}{Theorem}

\renewcommand{\AA}{\mathcal{A}}
\newcommand{\BB}{\mathcal{B}}

\newcommand{\E}{{\mathbb{E}}}
\newcommand{\PP}{{\mathbb{P}}}
\newcommand{\FF}{{\mathcal{F}}}

\newcommand{\Var}{\mathrm{Var}}

\newcommand{\1}{\mathbbm{1}}

\DeclareMathOperator{\esssup}{ess\,sup}
\DeclareMathOperator{\essinf}{ess\,inf}

\begin{document}


\title[On-Line Unimodal Selection]%
{Optimal Sequential Selection of a Unimodal Subsequence of a Random Sequence}
\author[Arlotto, A., and Steele, J. M.]{Alessandro Arlotto and J. Michael Steele}

\thanks{A. Arlotto:  Wharton School, Department of Operations and Information Management, Huntsman Hall
527.2, University of Pennsylvania, Philadelphia, PA 19104}

\thanks{J. M.
Steele:  Wharton School, Department of Statistics, Huntsman Hall
447, University of Pennsylvania, Philadelphia, PA 19104}

\begin{abstract} We consider the problem of selecting sequentially a \emph{unimodal subsequence}
from a sequence of independent identically distributed random variables, and we find that a person doing optimal
sequential selection does within a factor of the square root of two as well as a
prophet who knows all of the random observations in advance of any selections.
Our analysis applies in fact to selections of subsequences that have $d+1$ monotone blocks, and, by including the case
$d=0$, our analysis also covers monotone subsequences.

\end{abstract}
\maketitle

\section{Introduction}

A classic result of  \citeasnoun{ErdSze:CM1935} tells us that in any
sequence $x_1, x_2, \ldots , x_n$ of $n$ real numbers there is a subsequence of
length $k=\lceil n^{1/2} \rceil$ that is either monotone increasing or monotone decreasing.
More precisely, given $x_1, x_2, \ldots , x_n$ one can always find a subsequence $1\leq n_1 < n_2 < \cdots < n_k \leq n$ for which we either have
$$
x_{n_1} \leq x_{n_2} \leq \cdots \leq x_{n_k}, \quad \text{or}
\quad
x_{n_1} \geq x_{n_2} \geq \cdots \geq x_{n_k}.
$$
Many years later, Fan \citeasnoun{Chu:JCT1980} considered the analogous problem for unimodal sequences. Specifically, she sought to determine
the maximum value $\ell_n$ such that in any sequence of $n$ real values $x_1, x_2,
\ldots , x_n$ one can find a subsequence $x_{i_1}, x_{i_2}, \ldots, x_{i_k}$ of length $k=\ell_n$ and a
``turning place" $1\leq t \leq k$ for which one either has
$$
x_{i_1}\leq x_{i_2}\leq \cdots \leq  x_{i_t} \geq x_{i_{t+1}}\geq \cdots \geq x_{i_k}, \quad \text{or}
$$
$$
x_{i_1}\geq x_{i_2}\geq \cdots \geq  x_{i_t} \leq x_{i_{t+1}}\leq \cdots \leq x_{i_k}.
$$
Through a sustained and instructive analysis,
she surprisingly obtained an exact formula:
$$
\ell_n = \left\lceil \left(3n -3/4\right)^{1/2}-1/2\right\rceil.
$$
Shortly afterwards, \citeasnoun{Ste:DM1981} considered unimodal subsequences of permutations, or equivalently,
unimodal subsequences of a sequence of $n$ independent, uniformly distributed random variables
$X_1, X_2, \ldots,X_n$. For the random variables
\begin{align*}
U_n = \max\{k:\; & X_{i_1}\leq X_{i_2}\leq\cdots \leq  X_{i_t} \geq X_{i_{t+1}}\geq \cdots \geq X_{i_k}, \hbox{ where } \\
                 & 1\leq i_1<i_2<\cdots<i_k\leq n\},
\end{align*}
and
\begin{align*}
D_n = \max\{k:\; & X_{i_1}\geq X_{i_2}\geq\cdots \geq  X_{i_t} \leq X_{i_{t+1}}\leq \cdots \leq X_{i_k}, \hbox{ where } \\
                 & 1\leq i_1<i_2<\cdots<i_k\leq n\},
\end{align*}
it was established  that
\begin{equation}\label{ProphetExpectedValues}
\E\left[\max \{U_n,\, D_n \}\right] \sim \E[U_n] \sim \E[D_n] \sim 2 (2n)^{1/2} \quad \text{as } n \rightarrow \infty.
\end{equation}

Here we consider analogs of the random variables $U_n$, $D_n$ and $L_n=\max\{U_n,D_n\}$ but
instead of seeing the whole sequence all at once, one observes the variables sequentially. Thus,
for each $1 \leq i \leq n$, the chooser must decide at time $i$ when $X_i$ is
first presented whether to accept or reject $X_i$ as an element of the unimodal subsequence.
The sequential (or on-line) selection for the much simpler problem of a monotone subsequence
--- the analog of the original \citeasnoun{ErdSze:CM1935} problem ---
was considered long ago in \citeasnoun{SamSte:AP1981}.

\subsection*{Main Results}

We denote by $\Pi(n)$ the set of all feasible policies for the unimodal sequential
selection problem for $\{X_1,X_2, \ldots, X_n\}$ where these random variables are independent
with a common continuous distribution function $F$.
Given any feasible sequential selection policy $\pi_n \in \Pi(n)$, if we let
$\tau_k$ denote the index of the $k$'th selected element, then for each $k$
the value $\tau_k$ is a stopping time with respect to the increasing sequence of
$\sigma$-fields $\FF_i=\sigma \{X_1,X_2, ..., X_i\}$, $1\leq i \leq n$.
In terms of these stopping times, the random variable
\begin{align*}
U^o_n(\pi_n)=\max  \{ k:\; & X_{\tau_1} \leq  X_{\tau_2} \leq \cdots \leq X_{\tau_t}\geq X_{\tau_{t+1}} \geq \cdots \geq  X_{\tau_k}, \text{ where}\\
                       & 1\leq \tau_1 < \tau_2 < \cdots < \tau_k \leq n \},
\end{align*}
is the length of the unimodal subsequence that is selected by the policy $\pi_n$. For the moment,
we just consider unimodal subsequences that begin with an increasing piece and end with a decreasing piece;
either of these pieces is permitted to have size one.

For each $n$ there is a policy $\pi_n^* \in \Pi(n)$
that maximizes the expected length of the selected subsequence, and the main issue is to determine
the asymptotic behavior of this expected value. The answer turns out to have an informative
relationship to the off-line selection problem. A prophet with knowledge of the whole sequence before making his choices
will do better than an optimal on-line chooser, but he will only do better by a factor of $\sqrt{2}$.

\begin{Theorem}[Expected Length of Optimal Unimodal Subsequences]\label{th:main}
For each $n\geq 1$, there is a $\pi^*_n\in \Pi(n)$, such that
\begin{equation*}
\E[U^o_n(\pi^*_n)] = \sup_{\pi_n\in \Pi(n)}\E[U^o_n(\pi_n)],
\end{equation*}
and for such an optimal policy one has the upper bound
\begin{equation*}
\E[U^o_n(\pi^*_n)] < 2 n^{1/2}
\end{equation*}
and the lower bound
\begin{equation*}
2 n^{1/2} - 4 (\pi / 6)^{1/2} n^{1/4} - O(1) < \E[U^o_n(\pi^*_n)]
\end{equation*}
which combine to give the asymptotic formula
\begin{equation*}
   \E[U^o_n(\pi^*_n)]  \sim 2 n^{1/2} \quad \text{as $n\rightarrow \infty$}.
\end{equation*}
\end{Theorem}

In a natural sense that we will shortly make precise, the optimal policy $\pi_n^*$ is unique. Consequently, one can ask about the
\emph{distribution} of the length $U^o_n(\pi^*_n)$ of the subsequence that is selected by
the optimal policy, and there is a  pleasingly general argument that gives an upper bound for the variance. Moreover, that bound
is good enough to provide a weak law for $U^o_n(\pi^*_n)$.

\begin{Theorem}[Variance Bound] \label{tm:VarianceBound}
For the unique optimal policy $\pi^*_n \in \Pi(n)$, one has the bounds
\begin{equation}\label{selfboundIneq}
  \Var[U^o_n(\pi^*_n)] \leq \E[U^o_n(\pi^*_n)] < 2n^{1/2}.
\end{equation}
\end{Theorem}

\begin{Corollary}[Weak Law for Unimodal Sequential Selections]
For the sequence of optimal policies $\pi^*_n \in \Pi(n)$, one has the limit
$$
U^o_n(\pi^*_n)/\sqrt{n} \stackrel{p}{\longrightarrow} 2  \quad \text{as } n \rightarrow \infty.
$$
\end{Corollary}

\medskip
\noindent
{\sl Organization of the Proofs.}
\smallskip

The proof of Theorem \ref{th:main} comes in two halves. First, we show by an elaboration of an argument of \citeasnoun{Gne:JAP1999}
that  there is an \emph{a priori} upper bound for $\E[U^o_n(\pi_n)]$ for all $n$ and all $\pi_n \in \Pi(n)$.
This argument uses almost nothing about the structure of the selection policy beyond the fact from Section \ref{se:Interval-policies-optimal}
that it suffices to consider policies that are specified by acceptance \emph{intervals}.
For the lower bound we simply construct a good (but suboptimal) policy. Here there is an obvious candidate, but the
proof of its efficacy seems to be more delicate than one might have expected.

The proof of Theorem \ref{tm:VarianceBound} in Section \ref{se:proof-variance-bound} exploits a martingale
that comes naturally from the Bellman equation.
The summands of the quadratic variation of this martingale are then found to have a fortunate relationship to the probability that an observation
is selected. It is this ``self-bounding" feature that leads one to the bound \eqref{selfboundIneq} of the variance by the mean.

In Section \ref{se:extensions} we outline analogs of Theorems \ref{th:main} and \ref{tm:VarianceBound} for
subsequences that can be decomposed into $d+1$ alternating monotone blocks (rather than just two). If one takes $d=0$,
this reduces to the monotone subsequence problem, and in this case only the variance bound is new.
Finally, in Section \ref{se:final-connections} we comment briefly on two conjectures. These deal with a more
refined understanding of $ \Var[U^o_n(\pi^*_n)]$ and with the naturally associated central limit theorem.

\section{Mean Bounds: Proof of Theorem \ref{th:main}}\label{se:proof-asy-result}

Since the distribution $F$ is assumed to be continuous and since the problem is unchanged by replacing $X_i$ by its monotone transformation
$F^{-1}(X_i)$,
we can assume without loss of generality that
the $X_i$ are uniformly distributed on $[0,1]$. Next, we introduce two tracking variables. First, we let $S_i$ denote the value of the last element
 that has been selected up to and including time $i$. We then let $R_i$  denote an indicator variable that tracks the monotonicity
of the selected subsequence; specifically we set $R_i=0$ if the selections made up to and including time $i$ are increasing; otherwise
we set $R_i = 1$.

The sequence of real values
$\{ S_i: R_i = 0,\, 1\leq i \leq n \}$ is thus a monotone increasing sequence,  though of course not in the strict sense
because there will typically be long patches where the successive
values of $S_i$ do not change. Similarly,
$\{ S_i: R_i = 1,\, 1\leq i \leq n \}$ is  monotone decreasing sequence, and the full sequence $\{S_i: \, 1\leq i \leq n \}$ is a
unimodal sequence --- in the non-strict sense that permits ``flat spots." As a convenience for later formulas, we also set $S_0=0$ and $R_0=0$.

\subsection*{The Class of Feasible Interval Policies}

Here we will consider feasible policies that have acceptance sets that are given by intervals.
It is reasonably obvious that any optimal policy must have this structure,
but for completeness we give a formal proof of this fact in Section \ref{se:Interval-policies-optimal}.

Now, if the value $X_i$ is under consideration for selection,
two possible scenarios can occur: if $R_{i-1}=0$ (so one is in the ``increasing part'' of the selected subsequence)
then a selectable $X_i$ can be \emph{above or below} $S_{i-1}$. On the other hand, if $R_{i-1}=1$ (and one is in the ``decreasing part'' of the
selected subsequence), then any selectable $X_i$ has to be smaller than $S_{i-1}$. Thus, to specify a feasible interval policy,
we just need to specify for each $i$ an interval $[a,b] \subset [0,1]$ where we accept $X_i$ if $X_i \in [a,b]$  and
we reject it otherwise.
Here, the values of the end-points of the interval are functions of
$i$, $S_{i-1}$, and $R_{i-1}$. In longhand, we write the acceptance interval
as
$$
\Delta_i(S_{i-1}, R_{i-1}) \equiv [a(i, S_{i-1}, R_{i-1}), \, b(i, S_{i-1}, R_{i-1})].
$$

There are some restrictions on the functions $a(i, S_{i-1}, R_{i-1})$ and  $b(i, S_{i-1}, R_{i-1})$.
To make these explicit we consider two sets of functions, $\AA$ and $\BB$.
We say
$a\in\AA$ provided that $a:\{1,2,...,n\}\times [0,1]\times\{0,1\} \rightarrow [0,1]$ and
$$
0 \leq a(i,s,r) \leq s \quad \text{ for all } s \in [0,1],\, r \in \{0,1\} \text{ and } 1\leq i \leq n.
$$
Similarly, we say $b\in\BB$ provided that $b:\{1,2,...,n\}\times [0,1]\times\{0,1\} \rightarrow [0,1]$
and
$$
 s \leq b(i,s,0) \leq 1 \quad \text{ for all } s \in [0,1] \text{ and } 1\leq i \leq n;
$$
$$
0 \leq b(i,s,1) = s  \quad \text{ for all } s \in [0,1] \text{ and } 1\leq i \leq n.
$$
Together a pair $(a,b) \in \AA \times \BB$  defines an \emph{interval policy} $\pi_n \in \Pi(n)$ where we
accept $X_i$ at time $i$ if and only if
$
X_i \in \Delta_i(S_{i-1}, R_{i-1}).
$
We let $\Pi'(n)$ denote the set of feasible interval policies.

\subsection*{Three Representations}

First we note that for $S_i$ we have a simple update rule driven by whether $X_i$ is rejected or accepted:
\begin{equation*}
S_i =
\begin{cases}
S_{i-1} \quad &\text{if } X_i \notin \Delta_i(S_{i-1}, R_{i-1})\\
X_i \quad &\text{if } X_i \in \Delta_i(S_{i-1}, R_{i-1}).
\end{cases}
\end{equation*}
For the sequence $\{R_i\}$ the update rule is initialized by setting $R_0=0$; one should then note that
only one change takes place in the values of the sequence $\{R_i\}$. Specifically, we change to $R_i=1$  at the first $i$ such that
$S_i< S_{i-1}$, i.e. the first instance where we have a decrease in our sequence of selected values.
For specificity, we can rewrite this rule as
\begin{equation}\label{eq:cases-Ri}
R_i = \left\{
        \begin{array}{ll}
          1 & \hbox{if $ X_i \in \Delta_i(S_{i-1}, R_{i-1})$} \\
            & \hbox{and $S_{i-1}=\max\{S_k:\, 1\leq k\leq i\}$} \vspace{10pt}\\
          R_{i-1} & \hbox{otherwise.}
        \end{array}
      \right.
\end{equation}
Finally, using $\1(E)$ to denote the indicator function of the event $E$, we see by counting
the occurrences of the ``selection events" $X_i \in \Delta_i(S_{i-1}, R_{i-1})$,
that for each $1\leq k \leq n$ the number of selections made up to and including time $k$ is given by the sum of the indicators
\begin{equation}\label{eq:selection-central}
U^o_k(\pi_n) =\sum_{i=1}^k \1 \left(X _i \in \Delta_i(S_{i-1}, R_{i-1}) \right).
\end{equation}

\subsection*{Proof of the Upper Bound (An \emph{a priori} Prophet Inequality)}

The immediate task is to show that
for all $n\geq 1$ and all $\pi_n \in \Pi'(n)$, one has the inequality
\begin{equation}\label{ASinequality}
\E[U^o_n(\pi_n)] < 2 n^{1/2} .
\end{equation}
It will then follow from Proposition \ref{pr:interval-optimal} in Section \ref{se:Interval-policies-optimal}
that the bound \eqref{ASinequality} holds for all  $\pi_n \in \Pi(n)$.
We start with the representation \eqref{eq:selection-central} and then after
two applications of the Cauchy-Schwarz inequality we have
\begin{align*}
\E[ U^o_n(\pi_n)]&=\sum_{i=1}^n \E\left[ b(i, S_{i-1}, R_{i-1})-a(i, S_{i-1}, R_{i-1}) \right] \\
&\leq n^{1/2} \left\{ \sum_{i=1}^n \left(\E \left[ b(i, S_{i-1}, R_{i-1})-a(i, S_{i-1}, R_{i-1})\right] \right)^2 \right\}^{1/2} \\
&\leq n^{1/2} \left\{ \sum_{i=1}^n \E\left[\left(b(i, S_{i-1}, R_{i-1})-a(i, S_{i-1}, R_{i-1})\right)^2 \right] \right\}^{1/2}.
\end{align*}
The target bound \eqref{ASinequality} is therefore an immediate consequence of the following --- curiously general --- lemma.

\begin{Lemma}[Telescoping Bound]\label{lm:inequality-L2-norm}
For each $n\geq 1$ and for any strategy $\pi_n \in \Pi'(n)$, one has the inequality
\begin{equation}\label{gBound}
\sum_{i=1}^n \E\left[\left(b(i, S_{i-1}, R_{i-1})-a(i,S_{i-1}, R_{i-1})\right)^2\right] < 4.
\end{equation}
\end{Lemma}
\begin{proof}
We first introduce a bookkeeping function $g:[0,1]\times \{0,1\}\rightarrow [0,2]$ by setting
\begin{equation*}\label{eq:g}
g(s, r) = \left\{
                      \begin{array}{ll}
                        s, & \hbox{if $r=0$} \\
                        2-s, & \hbox{if $r=1$.}
                      \end{array}
                    \right.
\end{equation*}
Trivially $g$ is bounded by $2$, and we will argue
by conditioning and telescoping that the left side of inequality \eqref{gBound} is bounded
above by $2 \, \E\left[g(S_{n}, R_{n})\right]<4$.
Specifically, if we condition on $\FF_{i-1}$,
then the independence and uniform distribution of $X_i$ gives us, after a few lines of straightforward calculation, that
\begin{align*}
\E [g( S_{i}, &R_{i})  - g(S_{i-1}, 0)~|~\FF_{i-1}] \\
        & =\int_{a(i,S_{i-1}, 0)}^{S_{i-1}} (g(x,1) - S_{i-1})\, dx + \int_{S_{i-1}}^{b(i,S_{i-1}, 0)} (g(x,0) - S_{i-1})\, dx\\
        & =\frac{1}{2}\left(b(i,S_{i-1}, 0)- a(i,S_{i-1}, 0)\right)^2 \\
        & \quad + \left(S_{i-1}-a(i,S_{i-1}, 0)\right)\left(2-S_{i-1}-b(i,S_{i-1}, 0)\right).
\end{align*}
Since last summand is non-negative we have the tidier bound
\begin{equation}\label{eq:lb-1}
  \left(b(i,S_{i-1}, 0)- a(i,S_{i-1}, 0)\right)^2 \leq 2 \, \E [g( S_{i}, R_{i}) - g(S_{i-1}, 0)~|~\FF_{i-1}].
\end{equation}
\pagebreak
By an analogous direct calculation one also has the identity
\begin{align}
 \E [g(S_{i}, 1) - g(S_{i-1}, 1)~|~\FF_{i-1}]  & = \int_{a(i,S_{i-1}, 1)}^{S_{i-1}} (g(x,1) - g(S_{i-1},1))\, dx  \label{eq:lb-2} \\
 \nonumber                                           & = \frac{1}{2}\left(b(i,S_{i-1}, 1)- a(i,S_{i-1}, 1)\right)^2.
\end{align}
Since $R_{i-1}=1$ implies $R_i=1$, we can write $g(S_i,R_i) -g(S_{i-1}, R_{i-1})$ as the sum
$$
\{g(S_i,R_i) -g(S_{i-1}, 0)\} \1(R_{i-1}=0)+\{g(S_i,1) -g(S_{i-1}, 1)\} \1(R_{i-1}=1),
$$
so the two bounds \eqref{eq:lb-1} and \eqref{eq:lb-2} give us the key estimate
\begin{equation*}\label{eq:gS-telescope}
 \left(b(i,S_{i-1}, R_{i-1})- a(i,S_{i-1}, R_{i-1})\right)^2 \leq 2 \, \E [g( S_{i}, R_{i}) - g(S_{i-1}, R_{i-1})~|~\FF_{i-1}] .
\end{equation*}
Finally, when we take the total expectation and sum, one sees that telescoping gives
$$
\sum_{i=1}^n \E\left[\left(b(i, S_{i-1}, R_{i-1})-a(i,S_{i-1}, R_{i-1})\right)^2\right] \leq 2\, \E\left[g(S_{n}, R_{n})\right] < 4,
$$
just as needed.
\end{proof}

\subsection*{Proof of the Lower Bound (Exploitation of Suboptimality)}

We construct an explicit policy $\widetilde \pi_n \in \Pi(n)$ that is close enough to optimal to give us the bound
\begin{equation}\label{eq:lower-bound-EUopt}
2 n^{1/2} - 4 (\pi/6)^{1/2} n^{1/4} - O(1) < \E[U^o_n(\pi^*_n)].
\end{equation}
The basic idea is to make an approximately optimal choice of an increasing subsequence from the sample $\{X_i: 1 \leq i \leq n/2\}$
and an approximately optimal choice of a decreasing subsequence
from the sample $\{X_i: n/2 +1 \leq i \leq n\}$.  The cost of giving up a flexible choice of
the ``turn-around time" is substantial, but this
class of policies is still close enough to optimal to give required bound \eqref{eq:lower-bound-EUopt}.

For the moment, we assume that $n$ is even. We then select observations according to the following process:
\begin{itemize}
  \item For $ 1 \leq i \leq n/2$ we select the observation $X_i$ if and only if  $X_i$ falls in  the interval
  between $ S_{i-1}$ and $\min \{ 1, \, S_{i-1} + 2 n^{-1/2} \}. $
  \item We set $S_{n/2} = 1$ and for $ n/2 + 1 \leq i \leq n $ we select the observation $X_i$ if and only if $X_i$ falls in the
  interval between $\max \{ 0, \, S_{i-1} - 2 n^{-1/2} \}$ and $S_{i-1} . $
\end{itemize}
Here, of course, the selections for $1\leq i\leq n/2$ are increasing and the selections for $n/2+1\leq i\leq n$ are decreasing,
so the selected subsequence is indeed unimodal.

We then consider the stopping time
$$
\nu =  \min \{ i: S_i > 1- 2 n^{-1/2} \text{ or } i \geq n/2 \},
$$
and we note that the representation \eqref{eq:selection-central}, the suboptimality of the policy $\widetilde \pi_n$,
and the symmetry between our policy on $ 1 \leq i \leq n/2$ and on $ n/2+1 \leq i \leq n$
will give us the lower bound
\begin{equation}\label{eq:Unu}
2\, \E\left[\sum_{i=1}^{\nu} \1 \left(X _i \in [S_{i-1}, S_{i-1}+2n^{-1/2}] \right)\right]
\leq \E[U^o_n(\widetilde \pi_n)] \leq \E[U^o_n(\pi^*_n)].
\end{equation}
Wald's Lemma now tells us that
$$
\E\left[\sum_{i=1}^{\nu} \1 \left(X _i \in [S_{i-1}, S_{i-1}+2n^{-1/2}] \right)\right] = 2 \, n^{-1/2} \E [ \nu ], $$
so we have
$$
4 \, n^{-1/2} \, \E [ \nu ] \leq  \E[U^o_n(\pi^*_n)].
$$
The main task is to estimate $\E [ \nu ]$.
It is a small but bothersome point that the summands $\1 \left(X _i \in [S_{i-1}, S_{i-1}+2n^{-1/2}] \right)$
are not i.i.d. over the entirety of the range $i \in [1,n/2]$; the distribution of the
last terms differ from that of the predecessors.
To deal with this nuisance, we
take $Z_j$, $1\leq j< \infty$, to be a sequence of random variables defined by setting
\begin{equation*}
Z_j =  \left\{
             \begin{array}{ll}
               0 & \hbox{w.p. $1-2n^{-1/2}$} \\
               U_j & \hbox{w.p. $2n^{-1/2}$,}
             \end{array}
           \right.
\end{equation*}
where the $U_j$'s are independent and uniformly distributed on $[0, 2n^{-1/2}]$. Easy calculations now give us
for all $1\leq j<\infty$ that
\begin{equation}\label{varianceETC}
\E Z_j = \frac{2}{n},\quad \Var[Z_j] = \frac{8 n^{1/2} - 12}{3n^2} < \frac{8}{3 n^{3/2}}, \quad \text{and } |Z_j - \E Z_j| < \frac{2}{n^{1/2}}.
\end{equation}
Next, if we set $\widetilde S_0 \equiv 0$ and put
$$
\widetilde S_i = \sum_{j=1}^i Z_j, \quad \text{for }1\leq i \leq n,
$$
for $1\leq i \leq \nu$, we have $S_i \stackrel{d}{=}  \widetilde S_i$.
Setting
$ \widetilde\nu = \min \{ i: \widetilde S_i > 1- 2 n^{-1/2} \text{ or } i \geq n/2 \} $
we also have $\nu \stackrel{d}{=} \widetilde \nu$,
so to estimate $\E[\nu]$ it then suffices to estimate
\begin{equation*}\label{eq:Enu-tilde-tail-sum1}
\E [ \widetilde \nu ] =\!\!\! \sum_{i=0}^{n/2 - 1} \PP \left( \widetilde \nu > i \right)
=\!\!\!  \sum_{i=0}^{n/2 - 1}  \PP \left( \widetilde S_i
\leq 1 - 2 n^{-1/2} \right)=\frac{n}{2} - \!\! \sum_{i=0}^{n/2 - 1}  \PP \left( \widetilde S_i > 1 - 2 n^{-1/2} \right).
\end{equation*}
The proof of the lower bound \eqref{eq:lower-bound-EUopt} will then be complete once we check that
\begin{equation} \label{eq:tail-estimate}
\sum_{i=0}^{n/2 - 1}  \PP \left( \widetilde S_i > 1 - 2 n^{-1/2} \right) <  (\pi/6)^{1/2} n^{3/4} +  \lceil n^{1/2} \rceil.
\end{equation}
This bound turns out to be a reasonably easy consequence of Bernstein's inequality
\citeaffixed[Theorem 6]{Lug:LN2009}{c.f.,}
which asserts that for \emph{any} i.i.d sequence $\{Z_j\}$ with the almost sure bound $| Z_j - \E Z_j | \leq M$ one has for all $ t > 0$ that
\begin{equation*}
\PP \left( \sum_{j=1}^i \left\{ Z_j - \E Z_j \right\} > t \right)
\leq \exp\left\{ - \frac{t^2}{ 2 i \Var[Z_1] + 2 M t / 3}\right\}.
\end{equation*}
If we set $n^* = \lfloor n/2 - n^{1/2}- 1 \rfloor$, then Bernstein's inequality together with the bounds \eqref{varianceETC}
and some simplification will give us
\begin{align}
\sum_{i=0}^{n/2 - 1}  \PP \left( \widetilde S_i > 1 - 2 n^{-1/2} \right)
 & \leq  \lceil n^{ 1/2 } \rceil +  \sum_{i=0}^{ n^* }    \PP \left( \widetilde S_i > 1 - 2 n^{-1/2} \right)
   \nonumber \\
 & \leq \lceil n^{ 1/2 } \rceil +  \sum_{i=0}^{ n^* } \exp \left\{ - \frac{3 \left(-2 i - 2 n^{1/2} + n\right)^2}{8 n  \left( n^{1/2} - 1  \right)} \right\}.
 \nonumber
\end{align}
The summands are increasing, so the sum is bounded by

\begin{equation*}
\int_0^{ n/2 - n^{1/2}} \!\!\!\!\!\!\!\!\!\!\!
\exp \left\{ - \frac{3 \left(-2 u - 2 n^{1/2} + n\right)^2}{8 n  \left( n^{1/2} - 1  \right)} \right\} du
=(2/3)^{1/2}(n^{3/2} - n)^{1/2}  \int_0^{\alpha(n)} \!\!\! e^{- u^2} \, du,
\end{equation*}
where $\alpha(n) = (3/8)^{1/2} \left(n^{1/2} -2\right) (n^{1/2}-1)^{-1/2}$.
Upon bounding the last integral by $\pi^{1/2} / 2 $, one then completes the proof of the target bound \eqref{eq:tail-estimate}.
Finally, we note that if $n$ is odd, one can simply ignore the last observation at the cost of
decreasing our lower bound by at most one.

\medskip
\noindent
{\sl Remark.} A benefit of Bernstein's inequality (and the slightly sharper Bennett inequality) is that one gets to take advantage
of the good bound on $\Var[Z_j]$. The workhorse Hoeffding inequality would be blind to this useful information.

\section{Variance Bound: Proof of Theorem \ref{tm:VarianceBound}}\label{se:proof-variance-bound}

To prove the variance bound in Theorem \ref{tm:VarianceBound} we need some
of the machinery of the Bellman equation and dynamic programming. To introduce the
classical backward induction, we first set
$v_i(s,r)$ equal to the expected length of the longest unimodal
subsequence of $\{X_i, X_{i+1},\ldots, X_n\}$ that is obtained by
sequential selection when $S_{i-1}=s$ and $R_{i-1}=r$.
We then have the ``terminal conditions"
$$
v_n(s,0)=1, \quad v_n(s,1)=s, \quad \hbox{for all $s\in [0,1]$}
$$
and we set
$$
v_{n+1}(s,r)\equiv 0 \quad \hbox{for all $s\in [0,1]$ and $r\in\{0,1\}$}.
$$
For  $1\leq i\leq n-1$  we have the \emph{Bellman equation}:
\begin{equation}\label{eq:BellmanLUS}
v_i(s,r)= \left\{
                  \begin{array}{ll}
                    \int_0^{s} \max\left\{v_{i+1}(s,0),1+v_{i+1}(x,1)\right\}\,dx &\;\;\;\;\;\;\;\;\;\;\hbox{if $r=0$} \\
                    + \int_{s}^1 \max\left\{v_{i+1}(s,0),1+v_{i+1}(x,0)\right\}\,dx& \\
                                            & \\
                    (1-s)v_{i+1}(s,1)                                           & \;\;\;\;\;\;\;\;\;\;\hbox{if $r=1$}\\
                    +\int_0^{s} \max\left\{v_{i+1}(s,1),1+v_{i+1}(x,1)\right\}\,dx. & \\
                  \end{array}
                \right.
\end{equation}
One should note that the map $s \mapsto v_i(s,0)$ is continuous
and strictly decreasing on $[0,1]$ for $1\leq i\leq n-1$
with $v_n(s,0)=1$ for all $s\in [0,1]$. In addition, the map
$s \mapsto v_i(s,1)$ is continuous and strictly increasing on $[0,1]$ for all
$1\leq i\leq n$.

If we now define $a^*: \{1,2, \ldots, n\} \times [0,1]\times\{0,1\} \rightarrow [0,1]$ by setting
\begin{equation}\label{eq:optimal-threshold-down}
a^*(i, s, r)= \inf \left\{ x\in [0,s]:\, v_{i+1} (s, r) \leq 1+v_{i+1}(x, 1) \right\}, \\
\end{equation}
then we have $a^*\in\AA$.
Similarly, if we define $b^*: \{1,2, \ldots, n\} \times [0,1]\times\{0,1\} \rightarrow [0,1]$ by setting
\begin{equation}\label{eq:optimal-threshold-up}
b^*(i, s, r)= \left\{
                \begin{array}{ll}
                  \sup \left\{ x\in [s, 1]:\, v_{i+1} (s, 0) \leq 1+v_{i+1}(x, 0) \right\} & \hbox{if $r=0$.} \\
                  s & \hbox{if $r=1$.}
                \end{array}
              \right.
\end{equation}
then we have $b^*\in\BB$. Here, $a^*(i, s, r)$ and $b^*(i, s, r)$ are
state-dependent thresholds for which one is indifferent between (i)
selecting the current observation $x$, adjusting $r$ to $r'$ as in
\eqref{eq:cases-Ri},  and continuing to act optimally with new state pair $(x,r')$,
or (ii) rejecting the current observation, $x$, and continuing to act
optimally with unchanged state pair, $(s,r)$.

By the Bellman equation \eqref{eq:BellmanLUS} and the continuity and monotonicity properties of the
value function, the values $a^*$ and $b^*$
provide us with a unique acceptance interval for all $1\leq i\leq n$ and all pairs $(s,r)$.
The policy $\pi^*_n$ associated with $a^*$ and $b^*$ then accepts $X_i$ at time $1\leq i\leq n$
if and only if
$$
X_i \in \Delta^*_i(S_{i-1}, R_{i-1}) \equiv [a^*(i, S_{i-1}, R_{i-1}),\, b^*(i, S_{i-1}, R_{i-1})],
$$
where, as in Section \ref{se:proof-asy-result},
$S_{i-1}$ is the value of the last observation selected
up to and including time $i-1$, and $R_{i-1}$ tracks the direction of the monotonicity
of the subsequence selected up to and including time $i-1$.
In Section \ref{se:Interval-policies-optimal} we will prove that this policy is indeed
the unique optimal policy for the sequential selection of a unimodal subsequence.

We do not need a detailed analysis of $a^*$ and $b^*$, but it is useful to collect some facts.
In particular, one should note that
$a^*(i, s, r)=0$ whenever  $v_{i+1} (s,r)\leq 1$ and $b^*(i, s, 0)=1$ whenever $v_{i+1} (s,0)\leq 1$.
In addition, the difference $b^*(i, s, r) - a^*(i, s, r)$ provides us with an explicit bound
on the increments of the value function $v_i(s,r)$, as the following lemma suggests.

\begin{Lemma}\label{lm:Bellman-bounds}
For all $s\in [0,1]$, $r \in \{0,1\}$ and $1\leq i\leq n$, we have
\begin{equation}\label{eq:BellmanBound2}
0 \leq v_{i}(s,r)-v_{i+1}(s,r)\leq b^*(i, s, r) - a^*(i, s, r) \leq 1.
\end{equation}
\end{Lemma}

\begin{proof}
The lower bound is trivial and it follows by the fact that $v_{i}(s,r)$ is strictly decreasing in $i$
for each $(s,r) \in [0,1] \times \{0,1\}$.

For the upper bound, we first assume that $r=0$. Then, subtracting $v_{i+1}(s,0)$ on both sides of equation
\eqref{eq:BellmanLUS} when $r=0$ and using the definition of
$a^*$ and $b^*$, we obtain
\begin{align*}
    v_{i}(s,0)-v_{i+1}(s,0) & =  - (b^*(i, s, r) - a^*(i, s, r))v_{i+1}(s,0)\\
                            & +\int_{a^*(i, s, r)}^{s} (1+v_{i+1}(x,1))\,dx + \int^{b^*(i, s, r)}_{s} (1+v_{i+1}(x,0))\,dx.
\end{align*}
Recalling the monotonicity property of $s \mapsto v_{i+1}(s,r)$, we then have
\begin{align*}
    v_{i}(s,0)-v_{i+1}(s,0) & \leq - (b^*(i, s, r) - a^*(i, s, r))v_{i+1}(s,0)\\
    & +(s - a^*(i, s, r))(1+ v_{i+1}(s,1)) + (b^*(i, s, r)-s)(1+ v_{i+1}(s,0)),
\end{align*}
and since $v_{i+1}(s,1)\leq v_{i+1}(s,0)$, we finally obtain
$$
    v_{i}(s,0)-v_{i+1}(s,0) \leq  b^*(i, s, r) - a^*(i, s, r) \leq 1,
$$
as \eqref{eq:BellmanBound2} requires. The proof for $r=1$ is very similar and it is therefore omitted.
\end{proof}

We now come to the main lemma of this section.

\begin{Lemma}\label{lm:ISmartingale}
The process defined by
\begin{equation*}
Y_i=U^o_i(\pi^*_n)+v_{i+1} (S_i, R_i) \quad \text{for all } 0\leq i\leq n,
\end{equation*}
is a martingale with respect to the natural filtration
$\{\FF_i\}_{0\leq i\leq n}$.
Moreover, for the martingale difference sequence $d_i = Y_i - Y_{i-1}$ one has that
\begin{equation*}
| d_i | = |~ Y_i - Y_{i-1} ~| \leq 1 \quad \hbox{ for all $1\leq i\leq n$.}
\end{equation*}
\end{Lemma}

\begin{proof}
We first note that $Y_i$ is $\FF_i$-measurable and bounded. Then, from the definition of
$v_i(s,r)$ we have that $v_i(S_{i-1}, R_{i-1}) = \E\left[U^o_n(\pi^*_n) - U^o_{i-1}(\pi^*_n)~|~\FF_{i-1}\right]$.
Thus,
$$
Y_i = U^o_i(\pi^*_n) + \E\left[U^o_n(\pi^*_n) - U^o_{i}(\pi^*_n)~|~\FF_{i}\right] = \E\left[U^o_n(\pi^*_n) ~|~\FF_{i}\right],
$$
which is clearly a martingale.

To see that the martingale differences are bounded let
$$
W_i = v_{i+1}(S_{i-1}, R_{i-1}) - v_i(S_{i-1}, R_{i-1})
$$
represents the change in $Y_i$ if we do not
select $X_i$, and let
$$
Z_i = (1+ v_{i+1}(X_i, \1(X_i < S_{i-1})) - v_{i+1}(S_{i-1}, R_{i-1}))\1(X _i\in \Delta^*_i(S_{i-1}, R_{i-1}))
$$
represents the change when we do select $X_i$.
We then have that
$$
d_i = W_i + Z_i,
$$
and by our Lemma \ref{lm:Bellman-bounds} we know that $ -1 \leq  W_i  \leq 0 $.
Moreover, the definition of the threshold functions $a^*$ and $b^*$
and the monotonicity property of $s \mapsto v_{i+1}(s,r)$ give us that $ 0 \leq  Z_i  \leq 1$,
so that  $| d_i | \leq 1$, as desired.
\end{proof}

\subsection*{Final Argument for the Variance Bound}

For the martingale differences
$d_i= Y_i - Y_{i-1}$
we have
$$
Y_n - Y_0 = \sum_{i=1}^n d_i, \quad \text{and} \quad  \Var[Y_n] = \E\left[\sum_{i=1}^n d_i^2\right],
$$
and we also have the initial representation
$$
Y_0 = U^o_0(\pi^*_n) + v_{1}(S_0, R_0)= v_1(0, 0) = \E[U^o_n(\pi^*_n)]
$$
and the terminal identity
$$
Y_n = U^o_n(\pi^*_n) + v_{n+1}(S_n, R_n)= U^o_n(\pi^*_n).
$$
We now recall the decomposition $d_i = W_i + Z_i$ introduced in the proof of Lemma \ref{lm:ISmartingale},
where
$$
W_i = v_{i+1}(S_{i-1}, R_{i-1}) - v_i(S_{i-1}, R_{i-1})
$$
and
$$
Z_i = (1+ v_{i+1}(X_i, \1(X_i < S_{i-1})) - v_{i+1}(S_{i-1}, R_{i-1}))\1(X _i\in \Delta^*_i(S_{i-1}, R_{i-1})).
$$
Since $W_i$ is $\FF_{i-1}$ measurable, we have
$$
 \E\left[d_i^2~|~\FF_{i-1}\right] = \E\left[Z_i^2~|~\FF_{i-1}\right] + 2\, W_i\, \E\left[Z_i~|~\FF_{i-1}\right] + W_i^2.
$$
We also have $0 = \E\left[d_i~|~\FF_{i-1}\right] = W_i + \E\left[Z_i~|~\FF_{i-1}\right]$ so
\begin{equation}\label{diandWdef}
\E\left[d_i^2~|~\FF_{i-1}\right] = \E\left[Z_i^2~|~\FF_{i-1}\right] - W_i^2.
\end{equation}
Finally, from the definition of $Z_i$, $a^*$ and $b^*$ we obtain
\begin{align*}
\E\left[Z_i^2~|~\FF_{i-1}\right] & = \int^{b^*(i, S_{i-1}, R_{i-1})}_{a^*(i, S_{i-1}, R_{i-1})}
                                          \!\!\!\!\! \left(1+v_{i+1}(x, \1(x < S_{i-1}))-v_{i+1}(S_{i-1}, R_{i-1})\right)^2\,dx \\
                                 & \leq b^*(i, S_{i-1}, R_{i-1}) - a^*(i, S_{i-1}, R_{i-1}),
\end{align*}
since the integrand is bounded by $1$.
Summing \eqref{diandWdef}, applying the last bound, and taking expectations gives us
$$
\Var[U^o_n(\pi^*_n)] \leq  \sum_{i=1}^n \E\left[ b^*(i, S_{i-1}, R_{i-1}) - a^*(i, S_{i-1}, R_{i-1}) \right] = \E[U^o_n(\pi^*_n)],
$$
where the last equality follows from our basic representation \eqref{eq:selection-central}.

\section{Intermezzo: Optimality and Uniqueness of Interval Policies}\label{se:Interval-policies-optimal}

The unimodal sequential selection problem is a finite horizon Markov decision problem
with bounded rewards and finite action space, and for such a problem it is known that
there exists a non-randomized Markov policy $\pi^*_n$
that is optimal \citeaffixed[Corollary 8.5.1]{BerShr:AP1978}{c.f.}.
This amounts to saying that there exists an  optimal strategy $\pi^*_n$ such that for each $i$, $S_{i-1}$ and $R_{i-1}$,
there is a Borel set $D^*_i(S_{i-1}, R_{i-1})\subseteq [0, 1]$ such that $X_i$ is accepted if
and only if $X_i \in D^*_i(S_{i-1}, R_{i-1})$. Here we just what to show that  the Borel sets
$D^*_i(S_{i-1}, R_{i-1})$ are actually intervals (up to null sets).


Given the optimal acceptance sets $  D^*_i(S_{i-1}, R_{i-1}) $, $1 \leq i \leq n$,
we now set
\begin{equation*}
v_i(S_{i-1}, R_{i-1}) = \E\left[ \sum_{k=i}^n \1 (X_k \in D^*_k(S_{k-1}, R_{k-1})) ~ \big| \FF_{i-1} \right],
\end{equation*}
so we have the recursion
\begin{equation}\label{eq:Optimal-Markov-Recursion}
v_i(S_{i-1}, R_{i-1}) = \E\left[ \1 (X_i \in D^*_i(S_{i-1}, R_{i-1})) + v_{i+1}(S_{i}, R_{i}) ~ \big| \FF_{i-1} \right],
\end{equation}
and $v_i(s, r)$ is just the optimal expected number of selections made from the
subsample $\{X_i, X_{i+1},\ldots, X_n\}$ given that  $S_{i-1} = s$ and $R_{i-1} = r$. We then note that $v_n(s,0) = 1 $
for all $s \in [0,1]$, and one can check by induction on $i$
that the map $s \mapsto v_i(s,0)$ is continuous and strictly decreasing in $s$ for $1 \leq i \leq n - 1$. A similar argument also gives that
the map $s \mapsto v_{i}(s,1)$ is continuous and strictly increasing in $s$ for all $1 \leq i \leq n$.

If we now set
\begin{align*}
a(i, S_{i-1}, R_{i-1}) &= \essinf D_i(S_{i-1}, R_{i-1}) \quad \text{and} \\
b(i, S_{i-1}, R_{i-1}) &= \esssup D_i(S_{i-1}, R_{i-1}),
\end{align*}
then we want to show for all $1\leq i \leq n$ and all $(S_{i-1}, R_{i-1})$ that we have
$$
\PP ( \{D_i(S_{i-1}, R_{i-1})^c \cap [a(i, S_{i-1}, R_{i-1}), b(i, S_{i-1}, R_{i-1}) ] \} ) = 0.
$$

To argue by contradiction, we suppose that there is an $1 \leq i \leq n$ and an acceptance set
$D_i^* \equiv D_i^*(S_{i-1}, R_{i-1})$
that is not equivalent to an interval; i.e. we suppose
\begin{equation}\label{eq:contradiction-assumption}
\PP ( \{  D_i^{*c} \cap [a^*(i, S_{i-1}, R_{i-1}), b^*(i, S_{i-1}, R_{i-1}) ] \} ) > 0.
\end{equation}
We then consider the sets
$$
L_i = [0, S_{i-1}] \cap D_i^*  \quad \text{and}  \quad
U_i = [S_{i-1}, 1] \cap D_i^*,
$$
and we introduce the intervals
$$
\widetilde L_i = [S_{i-1} - | L_i |, \, S_{i-1}] \quad  \text{and} \quad
\widetilde U_i = [S_{i-1}, \, S_{i-1} + | U_i |],
$$
where $|A|$ denotes the Lebesgue measure of a set $A$.
The set $\widetilde D_i = \widetilde L_i \cup \widetilde U_i$
is also an interval and $| \widetilde D_i | = | D_i^* |$, so, if we can show that
\begin{equation}\label{eq:interval-beats-markov}
\E [ \1 (X_i \in  D_i^*) + v_{i+1}(S_{i}, R_{i})  ]
< \E[ \1 (X_i \in \widetilde D_i) + v_{i+1}(S_{i}, R_{i})  ],
\end{equation}
then the representation \eqref{eq:Optimal-Markov-Recursion} tells us
that policy $\pi^*_n$ is not optimal, a contradiction.

To prove the bound \eqref{eq:interval-beats-markov}, we note that
\begin{align*}
\E & \left[ \1 (X_i \in \widetilde D_i) + v_{i+1}(S_{i}, R_{i}) ~ \big| \FF_{i-1} \right]
- \E\left[ \1 (X_i \in  D_i^*) + v_{i+1}(S_{i}, R_{i})  \big| \FF_{i-1} \right] \\
& = \E  \left[  v_{i+1}(X_{i}, R_{i}) \1 (X_i \in \widetilde D_i) ~ \big| \FF_{i-1} \right]
- \E\left[   v_{i+1}(X_{i}, R_{i})\1 (X_i \in  D_i^*)  \big| \FF_{i-1} \right]
\end{align*}
since $\widetilde D_i$ and $D_i^*$ are $\FF_{i-1}$-measurable and
$\E  [ \1 (X_i \in \widetilde D_i)  | \FF_{i-1} ] = \E  [ \1 (X_i \in  D_i^*)  | \FF_{i-1} ]$.
By our construction, we also have the identities
\begin{equation}\label{eq:first-addend}
\E  \left[  v_{i+1}(X_{i}, R_{i}) \1 (X_i \in \widetilde D_i) ~ \big| \FF_{i-1} \right]
= \int_{\widetilde L_i} v_{i+1}(x, 1) \, dx + \int_{\widetilde U_i} v_{i+1}(x, 0) \, dx,
\end{equation}
and
\begin{equation}\label{eq:second-addend}
\E  \left[  v_{i+1}(X_{i}, R_{i}) \1 (X_i \in  D_i^*) ~ \big| \FF_{i-1} \right]
= \int_{ L_i} v_{i+1}(x, 1) \, dx + \int_{ U_i} v_{i+1}(x, 0) \, dx.
\end{equation}
Now since $| L_i | = | \widetilde L_i |$ implies that
$| \widetilde L_i \cap L_i^c | = | L_i \cap \widetilde L_i^c |$, we can write
\begin{align}
\int_{\widetilde L_i} v_{i+1}(x, 1) \, dx  - \int_{ L_i} v_{i+1}(x, 1) \, dx
& = \int_{\widetilde L_i \cap L_i^c} \!\!\!\!\! v_{i+1}(x, 1) \, dx  - \int_{ L_i \cap \widetilde L_i^c} \!\!\!\!\! v_{i+1}(x, 1) \, dx  \nonumber\\
& = ( \beta_i - \alpha_i ) | \widetilde L_i \cap L_i^c |, \label{eq:first-difference}
\end{align}
where $\alpha_i= \alpha_i (S_{i-1}, R_{i-1})$, and $\beta_i=\beta_i(S_{i-1}, R_{i-1})$
are chosen according to the mean value theorem for integrals.
The sets $\widetilde L_i \cap L_i^c$ and $ L_i \cap \widetilde L_i^c $ are almost surely disjoint
since $\widetilde L_i \cap L_i^c \subset [ S_{i-1} - |L_i| , S_{i-1}]$
and $ L_i \cap \widetilde L_i^c \subset [0, S_{i-1} - |L_i| ]$.
So, we find that $\alpha_i < \beta_i$ since $v_{i+1}(x,1)$ is strictly decreasing in $x$.

A perfectly analogous argument tells us that we can write
\begin{equation}\label{eq:second-difference}
\int_{\widetilde U_i} v_{i+1}(x, 1) \, dx  - \int_{ U_i} v_{i+1}(x, 1) \, dx  = (\delta_i - \gamma_i) | \widetilde U_i \cap U_i^{c} |,
\end{equation}
where $\gamma_i < \delta_i$ and $\gamma_i$ and $\delta_i$ depend on $(S_{i-1}, R_{i-1})$.
If we now set  $$ c_i (S_{i-1}, R_{i-1}) = \min\{ \beta_i - \alpha_i , \delta_i - \gamma_i\},$$ then
the identities \eqref{eq:first-addend} and \eqref{eq:second-addend} and the differences
\eqref{eq:first-difference} and \eqref{eq:second-difference}
give us the bound
$$
c_i (S_{i-1}, R_{i-1})| \widetilde D_i \cap D_i^{*c} |
 \! \leq \!
 \E \!  \left[  v_{i+1}(X_{i}, R_{i}) \1 ( \! X_i \in \widetilde D_i \!) \! - \! v_{i+1}(X_{i}, R_{i})\1 ( \! X_i \in  D_i^* \!)  \big| \FF_{i-1} \right].
$$
Since $c_i (S_{i-1}, R_{i-1}) > 0$, the assumption \eqref{eq:contradiction-assumption}
implies that the left hand-side above is strictly positive. When we take
total expectation we get
$$
 0 < \E\left[  v_{i+1}(X_{i}, R_{i}) \1 (X_i \in \widetilde D_i) -  v_{i+1}(X_{i}, R_{i})\1 (X_i \in  D_i^*) \right].
$$
In view  of the recursion \eqref{eq:Optimal-Markov-Recursion}, this
contradicts the optimality of $\pi^*$. This completes the proof of \eqref{eq:interval-beats-markov}, and, in summary
we have the following proposition.

\begin{Proposition}\label{pr:interval-optimal}
If $\pi^*_n$ is an optimal non-randomized Markov policy for the unimodal sequential selection problem, then, up to sets of measure zero,
$\pi^*$ is an interval policy.
\end{Proposition}

\begin{Corollary}\label{pr:interval-unique}
There is a unique policy $\pi^*_n \in \Pi(n)$ that is optimal.
\end{Corollary}

To prove the corollary one combines  the optimality of the interval policy given by Proposition
\ref{pr:interval-optimal} with the monotonicity properties of the Bellman equation \eqref{eq:BellmanLUS}.
Specifically, the map $s \mapsto v_{i} (s,0)$ is strictly decreasing in $s$ for all $1 \leq i \leq n-1$
and the map $s \mapsto v_{i} (s,1)$ is strictly increasing in $s$ for all $1 \leq i \leq n$,
so the equations
\eqref{eq:optimal-threshold-down} and \eqref{eq:optimal-threshold-up} determine the
values $a^*( \cdot)$ and $b^*( \cdot )$ uniquely.

\section{Generalizations and Specializations: $d$-Modal Subsequences}\label{se:extensions}

There are natural analogs of Theorems \ref{th:main} and \ref{tm:VarianceBound} for
``$d$-modal subsequences," by which we mean  subsequences that are allowed to make
``$d$-turns'' rather than just one. Equivalently these are subsequences that are the concatenation of (at most)
$d+1$ monotone subsequences. If we let $U_{n}^{o,d}(\pi^*_n)$ denote the analog of $U^o_n(\pi^*_n)$ when the selected subsequence
is $d$-modal, then the arguments of the preceding sections may be adapted to provide information on the
expected value of $U_{n}^{o,d}(\pi^*_n)$ and its variance. Here one should keep in mind that the case $d=0$ is \emph{not} excepted;
the arguments of the preceding sections do indeed apply to the selection of monotone subsequences.

\begin{Theorem}[Expected Length of Optimal $d$-Modal Subsequences]
If $\Pi(n)$ denotes the class of feasible policies for the $d$-modal
subsequence selection problem, then there is a unique $\pi^*_n \in \Pi(n)$ such that
\begin{equation*}
\E[U^{o,d}_n(\pi^*_n)] = \sup_{\pi_n \in \Pi(n)}\E[U^{o,d}_n(\pi_n)].
\end{equation*}
Moreover, for all  $n\geq 1$  and $d \geq 0$ one has
\begin{equation}\label{eq:mean-bounds-d-unimodal}
c(d)^{1/2}  n^{1/2} -  c(d)^{3/4} (\pi / 3)^{1/2} n^{1/4} - O(1) < \E[U^{o,d}_n(\pi^*_n)] < c(d)^{1/2}  n^{1/2},
\end{equation}
where $c(d) = 2(d + 1)$.
In particular, one has
\begin{equation*}
   \E[U^{o,d}_n(\pi^*_n)]  \sim \{2(d + 1)\}^{1/2}  n^{1/2} \quad \text{as $n\rightarrow \infty$}.
\end{equation*}
\end{Theorem}

One should note that the case $d=0$ corresponds to the \emph{monotone
subsequence selection problem} studied by \citeasnoun{SamSte:AP1981}
and more recently by \citeasnoun{Gne:JAP1999}. The monotone selection problem is also equivalent to
certain bin packing problems studied by
\citeasnoun{BruRob:AAP1991} and \citeasnoun{RheTal:JAP1991}.

In the special case of $d=0$, our upper bound \eqref{eq:mean-bounds-d-unimodal} agrees with that of
\citeasnoun{BruRob:AAP1991} as well as with the result of \citeasnoun{Gne:JAP1999}.
Our lower bound \eqref{eq:mean-bounds-d-unimodal} on the mean for $d=0$ turns out to be slightly worse than that of
\possessivecite{RheTal:JAP1991} since our constant for the $n^{1/4}$ term is $2^{3/4}(\pi / 3)^{1/2} \sim 1.72$,
while theirs is $8^{1/4}\sim 1.68$.

\medskip

For the $d$-modal problem, one can also prove the a variance bound that generalizes Theorem \ref{tm:VarianceBound}
in a natural way.

\begin{Theorem}[Variance Bound for $d$-Modal Subsequences] \label{th:variance-bound-d-modal}
For the unique optimal policy $\pi^*_n \in \Pi(n)$ one has the bound
\begin{equation*}
  \Var[U^{o,d}_n(\pi^*_n)] \leq \E[U^{o,d}_n(\pi^*_n)].
\end{equation*}
\end{Theorem}

Chebyshev's inequality and
Theorem \ref{th:variance-bound-d-modal} now combine as usual to provide a weak law for $U^{o,d}_n(\pi^*_n)$.
Even for $d=0$ this variance bound is new.

\section{Two Conjectures}\label{se:final-connections}

Numerical studies for small $d$ and moderate $n$, support the conjecture that one has the asymptotic relation
\begin{equation}\label{eq:conjecture-variance}
\Var[U^{o,d}_n(\pi^*_n)] \sim \frac{1}{3} \E[U^{o,d}_n(\pi^*_n)] \quad \text{as } n \rightarrow \infty.
\end{equation}
As observed by an anonymous reader, the methods of Section \ref{se:proof-variance-bound} and the
concavity of the value function established in \citeasnoun{SamSte:AP1981} are in fact enough to prove an appropriate lower bound
\begin{equation}\label{d0var}
\frac{1}{3}\,\E[U^{o,d}_n(\pi^*_n)] - 2 < \Var[U^{o,d}_n (\pi^*_n)] \quad \text{where } d=0.
\end{equation}
Here one should now be able to prove an upper bound on $\Var[U^{o,d}_n (\pi^*_n)]$  that is strong enough to
establish the case $d=0$ of the conjecture \eqref{eq:conjecture-variance}, but confirmation of this has eluded us.

Also, by numerical calculations of the optimal policy $\pi_n^*$ and by subsequent simulations of $U^{o,d}_n(\pi^*_n)$ for $d=0$, $d=1$, and
modest values of $n$, it seems likely that the random variable $U^{o,d}_n(\pi^*_n)$
obeys a central limit theorem. Specifically, the natural conjecture is that for all $d\geq 0$ one has
\begin{equation}\label{CTLconj}
 \frac{\sqrt{3} \left(U^{o,d}_n(\pi^*_n) - \sqrt{ 2 (d+1) n } \right) }{  (2 (d+1) n  )^{1/4}}\, \, {\Longrightarrow}\, \,  N(0,1)
 \quad \text{as $n \rightarrow \infty$}.
\end{equation}
Implicit in this conjecture is the belief that the lower bound \eqref{eq:mean-bounds-d-unimodal} can be improved to
$\{ 2 (d+1) n \}^{1/2}- o(n^{1/4})$, or better.

So far, the only central limit theorem available for a sequential selection
problem is that obtained by \citename{BruDel:SPA2001} \citeyear{BruDel:SPA2001,BruDel:SPA2004}
for a Poissonized version of the monotone subsequence problem.
Given the sequential nature of the problem,
it appears to be difficult to de-Poissonize the  results of  \citeasnoun{BruDel:SPA2004}
to obtain conclusions about the distribution of $U^{o,d}_n(\pi^*_n)$  even for $d=0$.

For completeness, we should note that even for the \emph{off-line} unimodal subsequence problem, not much more
is known about the random variable $U_n$ than its asymptotic expected value \eqref{ProphetExpectedValues}.
Here one might hope to gain some information about the distribution of $U_n$ by the methods of
\citeasnoun{BolBri:AAP1992} and \citeasnoun{BolJan:CGP1997}, and it is even feasible --- but only remotely so --- that one could
extend the famous distributional results of \citeasnoun{BaiDeiKur:JAMS1999} to unimodal subsequences. More modestly,  one certainly should
be able to prove that the distribution of $U_n$ is \emph{not} asymptotically normal. One motivation for going after such a result would be
to underline how the restriction to sequential strategies can bring one back to the domain of the central limit theorem.

\subsection*{Acknowledgment:}
We are grateful to an insightful referee who outlined the proof of the bound \eqref{d0var} and who suggested the conjecture \eqref{eq:conjecture-variance}
for $d=0$.

\bibliographystyle{agsm}

\end{document}